\documentclass[a4paper,12pt]{article}
\usepackage[top=2.5cm,bottom=2.5cm,left=3cm,right=2cm]{geometry}
\usepackage{cite, amsmath, amssymb}
\usepackage[margin=1cm,font=small, format=hang,labelsep=period, labelfont=bf]{caption}
\usepackage{color}
\usepackage{graphics}
\usepackage{csquotes}
\usepackage{mathtools}
\usepackage{amsthm}
\usepackage{enumerate}
\usepackage{authblk}
\usepackage{breqn}
\usepackage{paralist}
\usepackage{amsmath}

\newcommand{\vk}{v_{qk,pk} }
\newcommand{\gk}{g_{qk,pk} }
\newcommand{\bj}{\bar\jmath}
\newcommand{\bi}{\bar\imath}
\newcommand{\be}{\begin{equation}}
\newcommand{\ee}{\end{equation}}

\newcommand{\vv}{{\bf V}}
\newcommand{\la}{\langle}
\newcommand{\ra}{\rangle}
\newcommand{\Supp}{\text{Supp}}

\newcommand{\np}{\mathbb{N}_+}

\newcommand{\al}{{\gamma}}

\newcommand{\xb}{{\bf x}}
\newcommand{\R}{\ensuremath{\mathbb{R}}}
\newcommand{\C}{\ensuremath{\mathbb{C}}}
\newcommand{\N}{\ensuremath{\mathbb{N}}}
\newcommand{\Q}{\ensuremath{\mathbb{Q}}}
\newcommand{\Z}{\ensuremath{\mathbb{Z}}}
\numberwithin{equation}{section}
\newtheorem{lm}{Lemma}
\newtheorem{theorem}{Theorem}
\newtheorem{pro}{Proposition}

\newtheorem{cor}{Corollary}

\newtheorem{definition}{Definition}

\numberwithin{equation}{section}

\title{Time-reversibility and integrability\\ of $p:-q$ resonant vector fields}

\author{Jaume Gin\'e$^1$, Valery G. Romanovski$^{2,3,4}$, Joan Torregrosa$^5$ \\
$^1${\it Departament de Matem\`atica, Universitat de Lleida, Av.~Jaume II, 69, 25001 Lleida, Catalonia, Spain}\\
$^2${\it Faculty of Electrical Engineering and Computer Science, University of Maribor,
Koro\v ska cesta 46, SI-2000 Maribor, Slovenia}\\
$^3${\it Center for Applied Mathematics and Theoretical Physics,\\
Mladinska 3, SI-2000 Maribor, Slovenia}\\
$^4${\it Faculty of Natural Science and Mathematics, University of Maribor,
Koro\v ska cesta 160, SI-2000 Maribor, Slovenia}\\
$^5${\it  Departament de Matem\'atiques, Universitat Aut\'onoma de Barcelona,  08193 Bellaterra, Barcelona, Catalonia, Spain}
}

\date{}

\begin{document}
\maketitle
\thispagestyle{empty}

\begin{abstract}
We study local analytical integrability in a neighborhood of $p:-q$ resonant
singular point of a two-dimensional vector field and its connection
to time-reversibility with respect to the non-smooth involution
$
\varphi(x,y)=(y^{p/q},x^{q/p}).
$
Some generalizations of the theory developed by K.~S.~Sibirsky for $1:-1$ resonant case
to the $p:-q$ resonant case are presented.
\end{abstract}

\maketitle

\section{Introduction}

 Consider  an $n$-dimensional  system of ordinary differential equations
 \be \label{sys_F}
 \dot x = F(x)
 \ee
  where  $F(x)$ is an $n$-dimensional vector functions
   defined on some domain $D$ of  $\R^n$ or $\C^n$.
It is said   (see e.g. \cite{BBT, Lamb})
that  system \eqref{sys_X}
  is \emph{time-reversible} on $D$ if there exists an
involution $\psi$ defined on $D$
 such that
\be \label{inv_gen}
D_\psi^{-1} \cdot  F \circ \psi = - F.
\ee
We say  that system \eqref{sys_X} is   \emph{completely analytically  integrable} on $D$  if it admits $n-1$ functionally independent analytic first integrals on $D$.

A time-reversal symmetry is one of the fundamental
symmetries that appears in nature, in  particular,  important both
for classical and quantum mechanics. Various properties of systems  exhibiting  such symmetries
have been studying by many authors, see e.g.
  \cite{AGG, BBT, GM, HPR, Lamb, TM11,WRZ}  and the references therein.

Our paper is devoted  to the investigation  of the interconnection of
time-reversibility and local integrability in a neighborhood
of a singular point of systems of the form
 \be \label{sys_X}
 \dot \xb =A \xb + X(\xb), 
 \ee
  where  $A$ is an $n \times n$ matrix with entries in $\R$ or $\C$
   $\xb=(x_1,\dots,x_n) $,  $ X(\xb)$ is a vector-function without
   constant and linear terms  defined on some domain $D$ of  $\R^n$ or $\C^n$.

One of the first results in such studies  is due to Poincar\'e.
It follows from his results  that if  in the two-dimensional
case the eigenvalue of $A$
are pure imaginary and the system has an axis  of symmetry
passing through the origin, then it admits an analytic  first integral
in a neighborhood of the origin.

A generalization  of this result is presented in \cite{Bib}, where it is shown that if {system \eqref{sys_X}} is time-reversibile with
respect to a certain linear involution and   two eigenvalues of the matrix $A$
are pure imaginary, then under some assumptions the
system has at least one analytic first integral in a neighborhood of the origin.

A detail study of the interconnection of time-reversibility and local integrability
for systems \eqref{sys_X} was presented in \cite{LPW}.
In \cite{LPW} and \cite{W} the notion of time reversibility was generalized to the
case when on the right hand side of \eqref{inv_gen} "$-1$" is replaced by a primitive
root of unity.


In this paper we limit our consideration to the two-dimensional systems
\eqref{sys_X} with non-degenerate matrix $A$.
In the case when a 2-dim system \eqref{sys_X} is real and the
eigenvalues of $A$ are pure imaginary, and the vector field is symmetric
with respect to a curve passing through the origin,  the origin  of \eqref{sys_X}
is a center, and, therefore, has an analytic local integral in a neighborhood
of the origin. This geometric argument was used in \cite{Z} in order to
find some integrable systems in the family of real  cubic systems
(see also \cite{BBT} for recent developments in this direction).
{The symmetry axis is, in the general case, an analytic curve passing through the origin. However, from the work of Montgomery and Zippin \cite{MZ}, any analytic involution $\psi$ associated with a time-reversal
symmetry can be linearized in such a way that the symmetry axis becomes a
straight line. From this result in \cite{AGG} the normal form theory is used to establish an algorithm to determine if
a two-dimensional systems \eqref{sys_X} is orbitally reversible.}

A detail study of real  polynomial  systems which are
time-reversible  under reflection with respect  to a line was performed
by Sibirsky \cite{Sib1,Sib2}. In particular, he showed that in the polynomial
case  the set of such systems in the space of parameters
is the  variety of a binomial ideal defined  by  invariants
of the rotation group of the system (some similar results were obtained also in \cite{CGMM,Liu2}).  Later on the results obtained
by Sibirsky were generalized to the case of complex systems \eqref{sys_X}
with $1:-1$ resonant singular point at the origin in \cite{JLR,R,RS1}.

In this paper we consider system  \eqref{sys_X}
having a  $p:-q$ resonant singular point at the origin, which we write
in the form
\begin{equation} \label{gspq}
\begin{aligned}
\dot x &= \phantom{-} p x - \sum_{j+k\ge 1, j\ge -1}^\infty
                                  a_{jk}x^{j+1}y^{k} =  px (1- \sum_{j+k\ge 1, j\ge -1}^\infty
                               \frac 1p   a_{jk}x^{j}y^{k} ) \\
\dot y &=            -q y + \sum_{ j+k\ge 1, j\ge -1  }^\infty
                                  b_{kj}x^{k}{y}^{j+1} =
                                   -qy (1- \sum_{ j+k\ge 1, j\ge -1  }^\infty
                                 \frac 1q b_{kj}x^{k}{y}^{j}  )  ,
\end{aligned}
\end{equation}
where $p,q\in \mathbb{N}$, $GCD(p,q)=1$.  Both vector field \eqref{gspq}
and the associated differential operator are  denoted by $\mathcal{X}$.

Unless $p=q=1$ this system is not time-reversible under a linear transformation.
Our  study deals with the  time-reversibility of system \eqref{gspq} with respect to
the involution
\be \label{inv_pq}
\varphi(x,y)=(y^{p/q},x^{q/p}).
\ee

We will prove that if a system \eqref{gspq} is time-reversible with respect to \eqref{inv_pq}, then it admits an analytic first integral on a neighborhood of the origin.
We will also  give  an extension of the results  of \cite{Sib1,Sib2} and their
generalizations obtained in \cite{JLR,R,RS1} presenting an algorithm for
finding  subsets of polynomial   systems \eqref{gspq}  which are time-reversible with respect to \eqref{inv_pq} and generalizing the notion of Sibirsky ideal to
polynomial systems \eqref{gspq}. It will be shown  that, in fact,
the theory developed for $1:-1$ resonant case can be extended to systems
\eqref{gspq}, however not to the whole family, but only to a certain subfamily of \eqref{gspq}.

\section{Integrability of time-reversible systems}

For system \eqref{gspq} it is always possible to find
a series of the form \eqref{pqInt}
\begin{equation} \label{pqInt}
\Psi(x, y) = x^q  y^p
           + \sum_{\substack{j + k > p + q \\ j, k \in \N_0}}
           v_{j - q,k - p} x^j y^k
\end{equation}
for which
\begin{equation} \label{pqbar4}
\mathcal{X} \Psi = g_{q,p}(x^qy^p)^2 + g_{2q,2p}(x^qy^p)^3 + g_{3q,3p}(x^qy^p)^{4}
                             + \cdots\,,
\end{equation}
where $g_{kq,kp}$, $k=1,2,\dots$, are polynomials in the parameters $a_{jk}, b_{kj}$
of system \eqref{gspq}. Polynomials   $g_{kq,kp}$ are called the  saddle
quantities of system \eqref{gspq} (sometimes also the  focus quantities).
 System \eqref{gspq} corresponding to  some fixed values
$a_{jk}^*, b_{kj}^*$  of the parameters   has a local analytical first integral in
a neighborhood of the origin if and only if    $g_{kq,kp}(a^*,b^*)=0$   $\forall k\in \N$
(see e.g. \cite{RS1,RXZ}).

The following theorem shows that if system \eqref{gspq} is time-reversible with respect
to \eqref{inv_pq} then it has an analytic first integral of the form
\eqref{pqInt}. Observe, that if $p=q=1$ then the map \eqref{inv_pq} is just a permutation
of the variables, so the statement presents  a generalization of known results of
\cite{Bib,LPW,R} to the case of $p:-q$ resonant systems.

\begin{theorem}\label{th1}
Assume that system \eqref{gspq}  is time-reversible
with respect to the involution \eqref{inv_pq},
that is,
\be \label{inv_yx}
  D_\varphi^{-1} \cdot  \mathcal{X}\circ  \varphi =  - \mathcal{X}.
\ee
Then  it admits an analytic first integral of the form
\eqref{pqInt} in a neighborhood of the origin.
\end{theorem}

To prove the theorem we will need the following results.

\begin{lm}\label{lem1}
System \eqref{gspq} with $p$ or $q$ different from 1 is time-reversible with respect to \eqref{inv_yx} if and only if
\be\label{abpq}
b_{qv,up}=\frac qp a_{qu,pv}, \qquad a_{qu,pv}=\frac pq  b_{qv,pu},
\ee
where $u, v=0,1,2,\dots$ and the other coefficients in \eqref{gspq} are equal to zero.
\end{lm}
\begin{proof}
Using involution \eqref{inv_pq}, that is, performing
the substitution
\be \label{subs_x1y1}
x_1=y^{p/q}, \qquad y_1 = x^{q/p},
\ee
after  straightforward calculations we obtain
\begin{equation} \label{gspq_11}
\begin{aligned}
\dot x_1 &= \phantom{}  -  p x_1 (1 - \sum_{j_1+k_1 \ge 1, j_1\ge -1}^\infty
                                \frac 1q  b_{k_1j_1}x_1^{\frac{q j_1}p}y_1^{\frac {pk_1}q }) \\
\dot y_1 &=          \phantom{-} q y_1 (1 - \sum_{ j_1+k_1\ge 1, j_\ge -1  }^\infty
                               \frac 1p    a_{j_1 k_1}x_1^{\frac {q k_1}p}{y_1}^{\frac{pj_1}q}).
\end{aligned}
\end{equation}

In view of \eqref{inv_yx} it should hold
\be \label{eq_ab}
\frac 1p a_{jk} x^j y^k =  \frac  1q  b_{k_1j_1}x^{\frac{q j_1}p}y^{\frac {pk_1}q }
\ee
where the exponents on the right hand side should be non-negative integers
or
$$
\frac{qj_1}{p}= \frac{pj_1}{q}=-1.
$$
However the latter equality is impossible unless $p=q=1$.
Thus, \eqref{eq_ab} can take place  if we set  $ j_1=p u,\ k_1=q v$, where
$u, v=0,1,2,\dots.$ This yields formulae \eqref{abpq}.
\end{proof}

The following statement follows directly from \eqref{abpq}.
\begin{cor}
If system \eqref{gspq} with $p$ or $q$ different from 1 is time-reversible with respect to \eqref{inv_yx} then $x=0$ and $y=0$ are the separatrixes of the system.
\end{cor}

{\it Remark.}  Formulas \eqref{abpq} and the results obtained below remain valid also
in the case of $1:-1$ resonant singular points. But since the results in the $1:-1$ resonant
case are known, below we work under the assumption $p/q\ne 1$ taking advantage from the fact
that in such case the subscripts of the parameters  $a_{jk}, \ b_{kj}$ of \eqref{gspq}
are non-negative.

We augment the set of coefficients in \eqref{pqInt} with the collection
\[
J = \{v_{-q+s,q-s} : s  =0, \dots, p+q\},
\]
where, in agreement with formula \eqref{pqInt}, we set $v_{00} = 1$ and
$v_{mn} = 0$ for all other elements of $J$, so that elements of $J$ are the
coefficients of the terms of degree $p + q$ in $\Psi(x,y)$ of the form \eqref{pqInt}.

By \cite[p. 117]{RS}
the  coefficients $v_{k_1,k_2}$ of series \eqref{pqInt}
 can be   computed recursively using  the formula
\be\label{vk1k2pq}
v_{k_1,k_2} =
\begin{cases}
\frac{1}{p k_1 - q k_2}
    {\displaystyle
    \sum_{\substack{s_1 + s_2 = 0 \\ s_1\ge -q, s_2 \ge -p}}^{k_1 + k_2 - 1}
    }
    [(s_1 + q) a_{k_1 - s_1, k_2 - s_2}
    -
     (s_2 + p) b_{k_1 - s_1, k_2 - s_2}]
    v_{s_1,s_2}
    &\text{if $ p k_1 \ne q k_2$} \\
0   &\text{if $p  k_1  = q k_2$,}
\end{cases}
\ee
(in \cite{RS} formula \eqref{vk1k2pq} was obtained for the case of polynomial
system \eqref{gspq} but, obviously, it remains valid also in the case
when the right hand sides of \eqref{gspq} are series).

We order the index set of parameters $a_{jk}$  in the first  equation of  \eqref{gspq} in some manner, say by
degree lexicographic order from least to greatest, and write the ordered set
 as $S = \{(1, 0), (0,1), (-1,2), (2,0),  \dots  \}$. Consistently with this
we then order the parameters as
$
(a_{10}, a_{01}, a_{-1,2}, a_{20}, \dots,
b_{02}, b_{2,-1},b_{10},b_{01})
$
so that any monomial appearing in $v_{ij}$ has the form
$
a_{10}^{\nu_1}  a_{01}^{\nu_2} \cdots a_{s, t}^{\nu_\ell}
b_{t, s}^{\nu_{\ell+1}} \cdots b_{10}^{\nu_{2\ell-1}} b_{01}^{\nu_{2\ell}}
$
for some $\nu = (\nu_1, \dots, \nu_{2\ell})$ and $\ell=1,2,\dots.$ To simplify the notation, for
$\nu \in \np^{2\ell}$  we write
\be \label{nu_gs}
[ \nu ]
\stackrel{\text{def}}{=}
a_{10}^{\nu_1}  a_{01}^{\nu_2} \cdots a_{s, t}^{\nu_\ell}
b_{t, s}^{\nu_{\ell+1}} \cdots b_{10}^{\nu_{2\ell-1}} b_{01}^{\nu_{2\ell}},
\ee
so if the $k$-th variable  in the product is $a_{pq}, $ then $2\ell -k+1$-st variable 
is $b_{qp}$.

For each  $m \in \mathbb{N}$ we consider the finite  subset $S_m$ of
the set $S$  which corresponds to the case when system \eqref{gspq} is
a polynomial system of degree $m$, so
$$
S_m = \{(1, 0), (0,1), (-1,2), (2,0),  \dots, (-1,m)  \}.
$$
Denote by $\ell(m)$ the number of elements in $S_m$
 and  let
$L^m : \np^{2\ell(m)} \to \Z^2$ be the linear map defined
by
\begin{multline} \label{Ldef}
L^m (\nu)
= (L_1^m (\nu), L^m_2(\nu))
      = \nu_1       (1, 0)  +  \nu_2 (0,1)  + \cdots \nu_{\ell(m)} (-1,m)+\\
       \nu_{\ell(m)+1} (m,-1)+  \cdots +  \nu_{2 \ell(m)-1}
         (1, 0)+  \nu_{2 \ell(m)}   (0, 1).
\end{multline}
Let  $k[a,b]$ be  the ring of polynomials in parameters $a_{jk}, b_{jk}$ of
system \eqref{gspq} over the field $k$
and for $f \in k[a,b]$ we write $f = \sum_{\nu \in \Supp(f)}f^{(\nu)}[\nu]$,
where $\Supp(f)$ denotes those $\nu \in \np^{2\ell(m)}$, $m=1,2,\dots ,$  that the
coefficient of $[\nu]$ in the polynomial $f$ is nonzero.


\begin{definition}\label{jkpoly}
For $(j,k) \in \N_{-1} \times \N_{-1}$, a polynomial
$$
f = \sum_{\nu \in \Supp(f)}f^{(\nu)}[\nu]$$
 in the polynomial ring  $\C[a,b]$ is a
$(j,k)$-\emph{polynomial} if, for every $\nu \in \Supp(f)$,
$L^\ell(\nu) = (j,k)$ for all sufficiently large $\ell$.
\end{definition}

The reader can consult  \cite[Section 3.4]{RS} for more details about $(j,k)$-\emph{polynomials}
in the case of polynomial system \eqref{gspq}.

From now on
we will limit our consideration to the  systems of the form
\begin{equation} \label{gs_uv_inf}
\begin{aligned}
\dot x &= \phantom{-} x (p  - \sum_{t_1+t_2= 1}^\infty
                                  a_{qt_1, pt_2 }x^{qt_1}y^{pt_2} ), \\
\dot y &=            -y( q - \sum_{ t_1+t_2= 1 }^\infty
                                  b_{qt_2,pt_1}x^{qt_2}{y}^{pt_1}).
\end{aligned}
\end{equation}

By Lemma \ref{lem1}, in the case when $p/q\ne 1$  systems \eqref{gspq}, which are time-reversible
with respect to involution \eqref{inv_pq},  form a subfamily
of systems \eqref{gs_uv_inf}, so we do not lose generality working with  family
\eqref{gs_uv_inf} if we are interesting in time-reversibility with respect to \eqref{inv_pq}.

\begin{lm}\label{lem2}
For system \eqref{gs_uv_inf}
if $v_{k_1,k_2} $ is a non-zero coefficient of  series \eqref{pqInt} computed by \eqref{vk1k2pq}, then
\be \label{t1t2}
k_1=t_1 q, \ k_2 =t_2 p
\ee
 for some non-negative integer $t_1, t_2$.
Moreover, under  involution \eqref{inv_yx} the term
\be \label{vt1t2}
v_{q t_1,p t_2} x^{q t_1+q} y^{p t_2+p}
\ee
of  \eqref{pqInt}
is changed to the term
\be \label{vt2t1}
v_{q t_2,p t_1} x^{q t_2+q} y^{p t_1+p},
\ee
and vice versa.
\end{lm}
\begin{proof}
For system \eqref{gs_uv_inf} the linear  map \eqref{Ldef} can be written in the
form
$$
L^{q m }(\nu)=(q L_1^m(\nu), p L_2^m(\nu)),
$$
where $(L_1^m(\nu),  L_2^m(\nu)$ is defined by \eqref{Ldef} and $L^{q m-q+1 }(\nu)=
\dots =  L^{q m-1 }(\nu)=  L^{q m }(\nu) $.
By Theorem 4 of \cite{RS_BMS}   $v_{k_1k_2}$
is a $(k_1,k_2)$ polynomial. Therefore for each monomial $[\nu]$  of
 $v_{k_1k_2}$ it holds that
 $$
(q L_1^m(\nu), p L_2^m(\nu)) =(k_1,k_2)
$$
for all sufficiently large $m$.
It means that $q$ divides $k_1$ and $p$ divides $k_2$, that is, \eqref{t1t2}
holds.

Performing in \eqref{vt1t2} substitution \eqref{subs_x1y1} we see that
\eqref{vt2t1} holds.
\end{proof}

{\it Remark.}  Theorem 4 of \cite{RS_BMS} mentioned above   was formulated in \cite{RS_BMS} for the case of polynomial systems \eqref{gspq} but it remains correct also in the case when the right
hand sides of \eqref{gspq} are series.

\begin{theorem}\label{th2}
  The formal series \eqref{pqInt}
computed according to \eqref{vk1k2pq} is unchanged under involution \eqref{inv_yx},
that is, in view of Lemma \ref{lem2},
\be \label{v_cond}
v_{q t_1,p t_2}= v_{q t_2,p t_1}.
\ee
\end{theorem}
\begin{proof}
We prove the claim using induction on $t_1+t_2$.
When $t_1=t_2=0$ we have by the  definition $v_{00}=1$, so the claim holds.

Using Lemma \ref{lem2}   we can write formula \eqref{vk1k2pq} as
\be\label{vk1k2pq_1}
v_{q t_1, p t_2} =
\begin{cases}
\frac{1}{pq  (t_1 - t_2)}
    {\displaystyle
    \sum_{\substack{s_1 + s_2 = 0 \\ s_1\ge -q, s_2 \ge -p}}^{q t_1 + p t_2 - 1}
    }
    [(s_1 + q) a_{q t_1 - s_1, p t_2 - s_2}
    -
     (s_2 + p) b_{q t_1 -  s_1, p t_2 - s_2}]
    v_{s_1,s_2}
    &\text{if $ t_1 \ne t_2$} \\
0   &\text{if $ t_1  = t_2$,}
\end{cases}
\ee

In view of \eqref{abpq} and taking into account that $v_{j,k}$  are $(j,k)$-polynomials
 for $t_1\neq t_2 $ we can change the rule of summation obtaining  from \eqref{vk1k2pq_1}
\be\label{vk1k2pq_2}
v_{q t_1, p t_2} =
\frac{1}{pq  (t_1 - t_2)}
    \sum_{s_1=0}^{q t_1}
   \sum_{s_2=0}^{p t_2}
    [(s_1 + q) a_{q t_1 - s_1, p t_2 - s_2}
    -
     (s_2 + p) b_{q t_1 -  s_1, p t_2 - s_2}]
    v_{s_1,s_2}=
\ee
\be \label{v_1}
\frac{1}{pq  (t_1 - t_2)}
    \sum_{\tilde s_1=0}^{t_1 }
   \sum_{\tilde s_2=0}^{ t_2}
    [(\tilde s_1 + 1) q a_{q (t_1 -\tilde s_1), p (t_2 -\tilde  s_2)}
    -
     (\tilde  s_2 + 1) p b_{q (t_1 - \tilde s_1), p (t_2 - s_2}]
    v_{q \tilde s_1,p \tilde s_2},
\ee
where $ s_1=q \tilde s_1,$ $  s_2=p \tilde s_2.$

Performing similar computations we have
\be\label{vk1k2pq_2_1}
v_{q t_2, p t_1} =
\frac{1}{pq  (t_2 - t_1)}
    \sum_{s_1=0}^{q t_2}
   \sum_{s_2=0}^{p t_1}
    [(s_1 + q) a_{q t_2 - s_1, p t_1 - s_2}
    -
     (s_2 + p) b_{q t_2 -  s_1, p t_1 - s_2}]
    v_{s_1,s_2}=
\ee
$$
\frac{1}{pq  (t_2 - t_1)}
    \sum_{\tilde s_2=0}^{t_2 }
   \sum_{\tilde s_1=0}^{ t_1}
    [(\tilde s_2 + 1) q a_{q (t_2 -\tilde s_2), p (t_1 -\tilde  s_1)}
    -
     (\tilde  s_1 + 1) p b_{q (t_2 - \tilde s_2), p (t_1 - \tilde s_1}]
    v_{q \tilde s_2,p \tilde s_1},
$$
where $ s_1=q \tilde s_2,$ $  s_2=p \tilde s_1.$

Using \eqref{abpq} we further obtain from \eqref{vk1k2pq_2_1}
$$
v_{q t_2, p t_1} =\frac{1}{pq  (t_2 - t_1)}
    \sum_{\tilde s_2=0}^{t_2 }
   \sum_{\tilde s_1=0}^{ t_1}
    [(\tilde s_2 + 1) p b_{q (t_1 -\tilde s_1), p (t_2 -\tilde  s_2)}
    -
     (\tilde  s_1 + 1) q a_{q (t_1 - \tilde s_1), p (t_2 - \tilde s_2}]
    v_{q \tilde s_2,p \tilde s_1}=
$$
\be \label{v_2}
\frac{1}{pq  (t_1 - t_2)}\sum_{\tilde s_1=0}^{ t_1}\sum_{\tilde s_2=0}^{t_2 }
[ (\tilde  s_1 + 1) q a_{q (t_1 - \tilde s_1), p (t_2 - \tilde s_2})-(\tilde s_2 + 1) p b_{q (t_1 -\tilde s_1), p (t_2 -\tilde  s_2)}
]  v_{q \tilde s_1,p \tilde s_2},
\ee
where we have changed  $v_{q \tilde s_2,p \tilde s_1}$
to $v_{q \tilde s_1,p \tilde s_2}$ using the induction hypothesis.
Comparing  the expressions for \eqref{v_1} and \eqref{v_2} we conclude that
\eqref{v_cond} holds, that is, the series   $\Psi(x,y)$
computed by \eqref{vk1k2pq} is unchanged under involution \eqref{inv_yx}.
\end{proof}

Using the obtained results we prove Theorem \ref{th1} as follows.

{\it Proof of Theorem \ref{th1}.} Denote by $\mathfrak{X}$ the vector field of
system \eqref{gs_uv_inf}. By Theorem \ref{th2} the series   $\Psi(x,y)$
computed by \eqref{vk1k2pq} is unchanged under involution \eqref{inv_pq}.

Assume that
$$
\mathfrak{X} \Psi =\alpha(x,y).
$$
Since by our assumption the system is time-reversible,
it also  holds that
$$
\mathfrak{X} \Psi =-\alpha(x,y),
$$
yielding $\alpha(x,y)\equiv 0$. That means,  $\Psi(x,y)$
is a formal first integral of \eqref{gspq}.
But then also there exists an analytic first integral of the form \eqref{pqInt}
(see e.g. \cite{RS_BMS}).  \hfill  $\square$

\begin{cor}
If system \eqref{gspq} is time-reversible  with respect to  involution \eqref{inv_pq},
then the system in the distinguished  Poincare-Dulac  normal form is also time-reversible  with respect to the same involution.
\end{cor}
\begin{proof}
Since by  Theorem \ref{th1} any time-reversible system \eqref{gspq} is locally analytically  integrable,  its distinguished  normal form can be written as
\begin{equation}\label{normal}
\dot x= p x(1+\sum_{k=1}^\infty  g_k (x^q y^p)^k), \qquad
 \dot y= -q y (1+\sum_{k=1}^\infty  g_k (x^q y^p)^k).
\end{equation}
(see e.g. \cite{RS_BMS,Z}). Clearly, the latter system is
time-reversible with respect to \eqref{inv_pq}.
\end{proof}

{The normal form of any locally analytically integrable system \eqref{gspq} is given by (\ref{normal}). System (\ref{normal}) is time-reversible with respect to the involution \eqref{inv_pq}. Therefore any locally analytically integrable system \eqref{gspq} is conjugate to a time-reversible system, in the sence that there exists a change of variables $\phi$ that transforms the original  
 system to the normal form (\ref{normal}) and consequently the original system is time-reversible with  respect to the involution $ \bar{\psi} = \phi^{-1} \circ \psi \circ \phi$, where $\psi$ is  involution \eqref{inv_pq}. Hence the analytically integrability of system \eqref{gspq} is always associated with a time-reversal symmetry. In fact, all the nondegenerate centers are conjugate to a time-reversible system, and all the nilpotent centers are orbitally  time-reversible. This  does not happens for systems with null linear part, see \cite{GM}.}

{The problem  with the map $\bar{\psi}$ is that we have no idea about  the form of $\bar{\psi}$ not even the leading terms of such involution. Therefore from the found  results  we cannot deduce an algorithm based on the computation of the involution of the original system. However several methods to compute the saddle or focus quantities are known, see for instance \cite{FG,GG,GL,RS} and references therein.} 

{Nevertheless, always there exists a change $\phi$ such that any differential system \eqref{gspq} is transformed to its normal form
\begin{equation}
\dot{y_1} = p y_1 (1 + Y_1 ( y_1^q y_2^p)), \qquad  \dot{y_2} = -q y_2 (1 + Y_2 ( y_1^q y_2^p)).
\end{equation}
Next going through the change of variables $z_1=y_1^q$ and  $z_2= y_2^p$ the normal form is transformed to the normal form of the resonance $1:-1$ and the results known for such resonance can be applied to the $p:-q$ resonance.}

\section{Conditions of time-reversibility and the Sibirsky ideal}

In this section we propose an algorithmic approach
which allows for a given polynomial family \eqref{gs_uv_inf}
 to find the set of  systems which are time-reversible with respect to
 \eqref{inv_pq}. We also give a description of the set using the so-called
 Sibirsky ideal obtaining    some generalizations of the results of  \cite{JLR,R}).

We will limit our consideration to polynomial  systems of the form \eqref{gs_uv_inf},
that is, systems of the form
 \begin{equation} \label{gs_uv}
\begin{aligned}
\dot x &= \phantom{-} x (p  - \sum_{u+v= 1}^n
                                  a_{qu, pv }x^{qu}y^{pv} ), \\
\dot y &=            -y( q - \sum_{ u+v= 1 }^n
                                  b_{qv,pu}x^{qv}{y}^{pu}),
\end{aligned}
\end{equation}
assuming that  $p/q \neq 1$.

Denote by $\ell $ the number of parameters in the first equation of \eqref{gs_uv}.
For $k=1,\dots, \ell$ let
\be
\label{zeta}
\zeta_k=u_k-v_k
\ee
 and
consider   the  ideal
\be \label{idealH}
 H = \langle
    1 - w \al, a_{q u_k, p v_k} - t_k, \
b_{q v_k, p u_k} - \frac qp \al^{\zeta_k  } t_k \ : \quad
  k =  1, \ldots, \ell
\rangle.
\ee

\begin{pro} \label{prTR} {The following statements hold:}\newline

1) The Zariski closure of the  set of systems  in family \eqref{gs_uv}, which are time-reversible with respect to involution
\eqref{inv_pq} after the
  transformation
\be \label{trans_ort}
x\to \alpha x, \ y\to \alpha^{-1} y
\ee
with $\alpha \in \C\setminus\{0\}$,
is the variety ${\bf V}(\mathcal{I})$   of the  ideal
  \be \label{Ir}
  \mathcal{I}=H\cap \C [ a,b].
  \ee 

2) If the parameters $ a_{qu, pv},\  b_{qv, pu}$ of system
\eqref{gs_uv}   belong to the variety ${\bf V}(\mathcal{I})$, then the system
admits a local analytic first integral of the form
\eqref{pqInt}.
\end{pro}
\begin{proof}
Performing in system \eqref{gs_uv} transformation \eqref{trans_ort}
we obtain the system of the same shape with the parameters
$  a_{qu, pv},  b_{qv, pu} $ changed according to the rule
$$
 a_{qu, pv} \mapsto  \alpha^{pv -qu}  a_{qu, pv},
 \qquad b_{qv, pu} \mapsto   \alpha^{pu-qv} b_{qv, pu},
$$
where $u+v=1,\dots,n$.

By Lemma \ref{lem1} the  system obtained after transformation \eqref{trans_ort}  is  time-reversible with respect to involution
\eqref{inv_pq} if and only if for some $\alpha\ne 0$
\be \label{conpq}
 \alpha^{pv -qu}  a_{qu, pv} = \frac{p}{q}  \alpha^{pu-qv} b_{qv, pu},
\ee
where $u+v=1,\dots,n$.
Equivalently, we can rewrite \eqref{conpq} as
\be \label{conpq_1}
a_{q u_k, p v_k} = t_k, \quad
b_{q v_k, p u_k} - \frac qp \al^{\zeta_k} t_k,
\ee
 where $\al = \alpha^{-(p+q)}  $,  $k=1,\dots,\ell$ and $\zeta_k$ are defined by \eqref{zeta}.

From \eqref{conpq_1} using   the  Implicitization Theorem (see e.g. \cite{Cox})
we conclude that the first statement of the proposition holds.

2) By the construction  ${\bf V}(\mathcal{I})$ is the Zariski closure of systems
which are time-reversible with respect to \eqref{inv_pq} after a linear transformation \eqref{trans_ort}, 
so, in view
of Theorem \ref{th1} it is the Zariski closure of systems which
admit a first integral of the form \eqref{pqInt}.
However the set of systems in the  space of parameters of
\eqref{gs_uv}
having an analytic  first integral integral
of the form \eqref{pqInt} is an algebraic set (see e.g. Theorem 3.2.5 of \cite{RS}).
Therefore all systems from ${\bf V}(\mathcal{I})$ admit an analytic  first integral 
of the form \eqref{pqInt}.
\end{proof}

{\it Remark.} Obviously, in generic case the set of  time-reversible systems
is a proper subset of  ${\bf V}(\mathcal{I})$.

\medskip

As an example we consider the $1:-2$ resonant system of the form \eqref{gs_uv} of degree five,
\begin{equation} \label{gs_uv5}
\begin{aligned}
\dot x &= \phantom{-} x (1  - \sum_{u+v= 1}^2
                                  a_{2u, v }x^{2u}y^{v} ) =
     x
     - a_{01} x y
     - a_{20} x^3 - a_{21} x^3 y - a_{02} x y^2
     - a_{40} x^5
     ,\\
\dot y &=            -y( 2 - \sum_{ u+v= 1 }^2
                                  b_{2v,u}x^{2v}{y}^{u})=
                 -2 y + b_{01} y^2
                  + b_{20} x^2 y + b_{21} x^2 y^2 + b_{02} y^3  + b_{40} x^4 y .
\end{aligned}
\end{equation}

\begin{pro}
System \eqref{gs_uv5}  admits an analytic first integral of the form
\eqref{pqInt} if the 10-tuple $(a_{01},\dots,a_{40},b_{40,\dots,b_{01}})$ of its coefficients belong to the variety of the ideal
\begin{multline*}
\widetilde {\mathcal{I}}=
\la
2 a_{21} - b_{21}, -a_{40} b_{01}^2 + 2 a_{20}^2 b_{02}, 4 a_{02} a_{40} - b_{02} b_{40},
8 a_{02} a_{20}^2 - b_{01}^2 b_{40}, -2 a_{02} a_{20} b_{20} +\\ a_{01} b_{01} b_{40},
2 a_{01} a_{40} b_{01} - a_{20} b_{02} b_{20},
4 a_{01} a_{20} - b_{01} b_{20}, -a_{02} b_{20}^2 + 2 a_{01}^2 b_{40},
8 a_{01}^2 a_{40} - b_{02} b_{20}^2\ra.
\end{multline*}
\end{pro}
\begin{proof}
In the case of system \eqref{gs_uv5} the ideal $H$ of Proposition \ref{prTR}
is
\begin{multline}
\la a_{01} - t_1,  a_{20} - t_2 , a_{02} - t_3 ,  a_{40} - t_4, a_{21}-t_5,\\  b_{20} - 2 t_1 \al^{-1},
 b_{01} - 2 t_2 \al , b_{40} - 2  t_3 \al^{-2} , b_{02} - 2 t_4 \al^{2},  b_{21}- 2 t_5, 1 - w \al
\ra .
 \end{multline}
Computing the reduced Groebner basis of this
ideal with respect to the lexicographic ordering with
$ \al> w> t_1> t_2> t_3> t_4>
 a_{01}> a_{02}> a_{20}> a_{21}> a_{40}> b_{01}> b_{02}> b_{20}> b_{21}> b_{40} $
  we obtain the set of polynomials
\begin{multline*}
\{2 a_{21} - b_{21},
4 a_{01} a_{20} - b_{01} b_{20}, a_{02} b_{20}^2 - 2 a_{01}^2 b_{40},
2 a_{02} a_{20} b_{20} - a_{01} b_{01} b_{40}, 8 a_{02} a_{20}^2 - b_{01}^2 b_{40},\\
a_{40} b_{01}^2 - 2 a_{20}^2 b_{02}, $ $ 2 a_{01} a_{40} b_{01} - a_{20} b_{02} b_{20},
8 a_{01}^2 a_{40} - b_{02} b_{20}^2,
4 a_{02} a_{40} -
 b_{02} b_{40}, 
\dots 
\},
\end{multline*}
where the dots stand for the polynomials which depend ob $\al,w, t_1,t_2,t_3,t_4,t_5$.

The polynomials of the Groebner basis 
 which do not
  depend on $\al,w, t_1,t_2,t_3,t_4$ form a basis of
   the  ideal $ \mathcal{I}$  of Proposition \ref{prTR}
  and they are exactly the polynomials  defining  the ideal in the statement of the
  present
  proposition.
\end{proof}

{\it Remark.}  By 1) of Proposition \ref{prTR} the variety ${\bf V}(\tilde {\mathcal{I}})$
is the Zariski closure of the set of time-reversible systems in family \eqref{gs_uv5}.

We denote  by $S$  the ordered  set of subscripts of the  coefficients of the nonlinear
terms of the first
equation in \eqref{gs_uv}. Letting       $\ell$  the number of elements of $S$,  $S$ can be 
 written as
 $$S =\{  (q u_1,p v_1),\dots, (q u_\ell, p v_{\ell})    \}= \{\bi_1, \dots, \bi_\ell\}.$$ For
$\bi_s = (q u_s, p v_s)$,  let $\bj_s = (q v_s, p u_s).$
We call $\bi_s$ and $\bj_s$ \emph{conjugate} vectors (or conjugate indexes).
Any monomial appearing in the coefficient
$v_{k_1 k_2}$  of \eqref{pqInt} has the form
$
a_{\bi_1}^{\nu_1} \cdots a_{\bi_\ell}^{\nu_\ell}
b_{\bj_\ell}^{\nu_{\ell+1}} \cdots b_{\bj_1}^{\nu_{2\ell}}
$
for some $\nu = (\nu_1, \dots, \nu_{2\ell})$. We use  notation \eqref{nu_gs}
adapted  to the case of system \eqref{gs_uv}, so now
\[
[ \nu ]
\stackrel{\text{def}}{=}
a_{\bi_1}^{\nu_1} \cdots a_{\bi_\ell}^{\nu_\ell}
b_{\bj_\ell}^{\nu_{\ell+1}} \cdots b_{\bj_1}^{\nu_{2\ell}}\,.
\]
For a given field $k$
we will write just $k[a,b]$ in place of
$k[a_{\bi_1}, \dots, a_{\bi_\ell}, b_{\bj_\ell}, \dots,  b_{\bj_1}]$,
and for $f \in k[a,b]$ write $f = \sum_{\nu \in \Supp(f)}f^{(\nu)}[\nu]$,
where $\Supp(f)$ denotes those $\nu \in \np^{2\ell}$ such that the
coefficient of $[\nu]$ in the polynomial $f$ is non-zero.

\begin{definition}\label{cdin}
Let
$$
f = \sum_{\nu \in Supp(f)} f^{(\nu)}
     a_{ qu_1,p v_1}^{\nu_1} \cdots a_{q u_\ell,p v_\ell}^{\nu_\ell}
     b_{q v_\ell,p u_\ell}^{\nu_{\ell+1}} \cdots b_{q v_1,p  u_1}^{\nu_{2\ell}}
\in \C[a,b].$$
The \emph{conjugate} $\widehat f$ of $f$ is the polynomial
obtained from $f$ by the involution
\[
f^{(\nu)} \to \bar f^{(\nu)}
\qquad
a_{q i, pj} \to b_{q j, p i}
\qquad
b_{q j, p i} \to a_{q i, p j};
\]
that is,
$$\widehat f =
\sum_{\nu \in Supp(f)} \bar f^{(\nu)}
a_{q u_1 , p v_1}^{\nu_{2\ell}} \cdots a_{ q u_\ell, p v_\ell}^{\nu_{\ell+1}}
b_{q  v_\ell,p  u_\ell}^{\nu_\ell} \cdots b_{q v_1,p u_1}^{\nu_1} \in \C[a,b].$$
\end{definition}

 Since
$
[\nu] = a_{q u_1, p v_1}^{\nu_1} \cdots a_{ q u_\ell,p v_\ell}^{\nu_\ell}
        b_{q v_\ell,p  u_\ell}^{\nu_{\ell+1}} \cdots b_{q v_1,p  u_1}^{\nu_{2\ell}}
$,
we have
$$\widehat{[\nu]}
=
a_{q u_1,p v_1}^{\nu_{2\ell}} \cdots a_{q u_\ell,p v_\ell}^{\nu_{\ell+1}}
b_{q v_\ell,p  u_\ell}^{\nu_\ell} \cdots b_{q v_1,p u_1}^{\nu_1},$$ so that
\begin{equation} \label{nuhatfact}
\widehat{[(\nu_1, \dots \nu_{2\ell})]} = [(\nu_{2\ell}, \dots, \nu_1)]\,.
\end{equation}
For this reason we will also write, for $\nu = (\nu_1, \dots, \nu_{2\ell})$,
$\widehat \nu = (\nu_{2\ell}, \dots, \nu_1)$.

Once the $\ell$--element set $S$ has been specified and ordered we let
$L : \np^{2\ell} \to \np^2$ be the map defined by
\be \label{L}
L(\nu) = (L_1(\nu), L_2(\nu))
       = \nu_1       \bi_1    + \cdots + \nu_\ell   \bi_\ell
       + \nu_{\ell+1}\bj_\ell + \cdots + \nu_{2\ell}\bj_1,
\ee
which is similar to the  map \eqref{Ldef}.

Let
$$
\mathcal{M}=\{\nu  \in \np^{2\ell} :\  L(\nu)=(qk,pk),\ k=0,1,2,\dots \}.
$$
Clearly,  $
\mathcal{M}$ is an Abelian monoid.

Let  $\mathfrak{X}$ be  the vector field of system \eqref{gs_uv}.
The following result was obtained in \cite{RS_BMS} (where speaking 
about $(s,t)$-polynomials we mean $(s,t)$-polynomials with respect to map \eqref{L}.
\begin{theorem}\label{th_int_pq}
Let family \eqref{gs_uv}  be given. There exists a formal
series $\Psi(x,y)$ of the form \eqref{pqInt} and polynomials $g_{q,p}$, $g_{2q,2p}$,
$\ldots$ in $\Q[a,b]$ such that
\begin{enumerate}
\item[(a)]
\be \label{g_qp}
\mathfrak{X}\Psi =\sum_{k=1}^\infty g_{qk,pk} x^{qk} y^{pk}
;
\ee
\item[(b)] for every pair $(i,j) \in \N_{-q}\times \N_{-p}$, $i + j \ge 0$,
           $v_{ij} \in \Q[a,b]$, and
$v_{ij}$ is an $(i,j)$--polynomial;
\item[(c)] for every $k \ge 1$, $\vk = 0$; and
\item[(d)] for every $k \ge 1$, $\gk \in \Q[a,b]$, and $\gk$ is a
           $(qk,pk)$--polynomial.
\end{enumerate}
\end{theorem}


For  a given  family \eqref{gs_uv} and ordered  set $S$ of indices
 for any $\nu \in \np^{2\ell}$ define $V(\nu) \in \Q$ recursively,
with respect to $|\nu| = \nu_1 + \cdots + \nu_{2\ell}$, as follows:
\[
V((0, \dots, 0)) = 1;
\]
for $\nu \ne (0, \dots, 0)$
\[
V(\nu) = 0 \qquad \text{if} \qquad p L_1(\nu) = q L_2(\nu);
\]
and
when $p L_1(\nu) \ne q  L_2(\nu)$\,,
\be \label{bar6}
\begin{aligned}
V(\nu)
=
&\frac{1}{p L_1(\nu) -q  L_2(\nu)} \times
\\
\Biggl[
&\sum_{j = 1}^\ell
  V(\nu_1, \dots, \nu_j - 1, \dots, \nu_{2\ell})
 (L_1(\nu_1, \dots, \nu_j - 1, \dots, \nu_{2\ell}) + q)   \\
-
&\sum_{j = \ell + 1}^{2\ell}
  V(\nu_1, \dots, \nu_j - 1, \dots, \nu_{2\ell})
 (L_2(\nu_1, \dots, \nu_j - 1, \dots, \nu_{2\ell}) + p)
 \Biggr],
\end{aligned}
\ee
where  $L(\nu)$ is defined by \eqref{L}.


\begin{theorem} \label{th_V} For a family of systems of the form \eqref{gs_uv}
 let $\Psi$ be the formal series of the form \eqref{pqInt} computed by \eqref{vk1k2pq},
 $\{ g_{q k, p k} : k \in \N \}$ be the polynomials   in $\C[a,b]$ satisfying \eqref{g_qp}. Then
\begin{compactenum}[(a)]
\item for $\nu \in \Supp(v_{k_1,k_2})$, the coefficient
      $v_{k_1,k_2}^{(\nu)}$ of $[\nu]$ in $v_{k_1,k_2}$ is $V(\nu)$,
\item for $\nu \in \Supp(g_{qk,pk})$, the coefficient $g_{qk,pk}^{(\nu)}$ of
      $[\nu]$ in $g_{qk,pk}$
      is
           \be\label{Gcoef}
          \begin{aligned}
           g_{qk, p k}^{(\nu)}\phantom{2}
           & =-\left[\sum_{j = 1}^\ell
            V(\nu_1, \dots, \nu_j - 1, \dots, \nu_{2\ell})
           (L_1(\nu_1, \dots, \nu_j - 1, \dots, \nu_{2\ell}) + q) \right.\\
           -
           &\left.\sum_{j = \ell + 1}^{2\ell}
            V(\nu_1, \dots, \nu_j - 1, \dots, \nu_{2\ell})
           (L_2(\nu_1, \dots, \nu_j - 1, \dots, \nu_{2\ell}) + p)\right]\,,
           \end{aligned}
           \ee
      and
\item the following identities hold:
           \begin{subequations}\label{bar7}
           \begin{align}
         \kappa   V(\widehat\nu) &=  V(\nu)
           \quad \text{and} \quad
          \kappa  g_{qk,pk}^{(\widehat\nu)} = - g_{qk,pk}^{(\nu)}
           \qquad
           \text{for all}
           \quad
           \nu \in \np^{2\ell} \label{bar7i} \\
           V(\nu) &= g_{qk,pk}^{(\nu)} = 0
           \quad
           \text{if}
           \quad
           \widehat \nu = \nu \ne (0, \ldots, 0) \label{bar7ii} \,, \\
           &
           \text{where} \nonumber
           \\
            \kappa & =\left( \frac{q}{p} \right)^{(\nu_1+\dots +\nu_\ell)-(\nu_{l+1}+\dots +\nu_{2\ell}) }.
           \end{align}
           \end{subequations}
\end{compactenum}
\end{theorem}

Statements (a) and (b) of the theorem are proved in \cite{RS_BMS}, statement
(c) can be proved similarly as statement 3) of Theorem 3.4.5 of \cite{RS}.

The next statements follows immediately from c) of Theorem  \ref{th_V}. 
\begin{cor}
The saddle quantities $g_{qk,pk}$ of system \eqref{gs_uv}
have the  form
\be \label{g_sq}
g_{qk,pk}= \frac 12 \sum_{v: L(\nu)=(qk,pk)} g^{(\nu)}_{qk,pk} (\kappa [\nu]- [\hat \nu]).
\ee
\end{cor}

By the  analogy with $1:-1$ resonant case we call the ideal
\be \label{id_S_def}
I_{Sib} =\la  \kappa [\nu]- [\hat \nu] : \nu \in \mathcal{M} \ra
\ee
the {\it Sibirsky ideal} of system \eqref{gs_uv}. Obviously, transformations \eqref{trans_ort} form a group.
It is easy to see that for any $\nu\in \mathcal{M}$
$[\nu]$ is an invariant of group \eqref{trans_ort}. Sibirsky
studied such invariants for the case of $1:-1$ resonant system \eqref{gspq}
and used the ideal \eqref{id_S_def} to describe  the basis of the invariants
and the number of axis of symmetry of the corresponding real systems \cite{Sib1,Sib2}.

\begin{theorem}\label{th_S}
If the $2 \ell$-tuple of  parameters of \eqref{gs_uv} belong to
$\vv (I_{Sib})$ then the corresponding  system admits an analytic  first integral
of the form \eqref{pqInt}.
\end{theorem}
\begin{proof}
The conclusion follows from formula \eqref{g_sq}.
\end{proof}



The following statement shows that the variety of the Sibirsky ideal is the
Zariski closure of the set of systems, which are  time-reversible with respect to
\eqref{inv_pq}. The proof is based on an adaption of the ideas of \cite{CT,St}.
\begin{theorem}
Let $\mathcal{I}$ be the ideal defined by \eqref{Ir}. Then  
\be\label{t.isym}
  I_{Sib}=\mathcal{I}.
 \ee
\end{theorem}
\begin{proof}
For $k=1,\dots, \ell$ let, as above,  $\zeta_k=u_k-v_k$ and
consider   the ring homomorphism
$$
\theta : \Q(a,b,t_1,\dots,t_\ell,\al,w) \longrightarrow
\Q(\al,t_1,\dots,t_\ell)
$$
defined by
\be\label{theta}
        a_{q u_k, p v_k} \mapsto t_k, \
b_{q v_k, p u_k} \mapsto \frac qp \al^{\zeta_k } t_k,\
w\mapsto 1/\al, \quad
  (k =  1, \ldots, \ell) .
  \ee

Let $H$ be the ideal \eqref{idealH}.
 Clearly,
  \be \label{Hker} H = \ker(\theta). \ee

 A reduced Groebner basis $G$ of
$\Q[a,b]\cap H$ can be found   computing a reduced Groebner basis of $H$ using an
elimination ordering with $\{a_{qu_j, p v_j},\ b_{q v_j, p u_j}\}<\{w, \al, t_j\} $
for all
$j=1,\dots, \ell$, and then intersecting it with $\Q[a,b]$.  Since
$H$ is binomial,  any reduced Groebner basis $G$ of $H$ also
consists of binomials. This means that $\mathcal{I}=H\cap \Q[a,b] $ is a
binomial ideal.

We show that $I_{Sib}\subset \mathcal{I}.$  Taking into account that
$\zeta_k=- \zeta_{2\ell -k}
$ by \eqref{theta} for any
$\alpha=  (\alpha_1,\dots,\alpha_{2\ell})  \in \mathcal{M}$
we have
$$
\begin{aligned}
\theta([\alpha])&  =
 t_1^{\alpha_1} \cdots t_\ell^{\alpha_\ell}
 t_\ell^{\alpha_{\ell+1}} (\frac qp)^{\alpha_{\ell+1}}  \al ^{\zeta_\ell
\alpha_{\ell+1}} \cdots t_1^{\alpha_{2\ell}} (\frac q p)^{\alpha_{2 \ell} }   \al ^{{\zeta_1}
\alpha_{2 \ell}} =\\
&
 (\frac qp)^{\alpha_{\ell+1}+\dots+\alpha_{2\ell} }
 t_1^{\alpha_1} \cdots t_\ell^{\alpha_\ell}
 t_\ell^{\alpha_{\ell+1}}
 \cdots t_1^{\alpha_{2\ell}} \al ^{-(\zeta_{\ell+1} \alpha_{\ell+1}+\dots+\zeta_{2\ell }
 \alpha_{2\ell} )}
 \end{aligned}
$$
and
$$
\begin{aligned}
\theta([\hat \alpha])&  =
 t_1^{\alpha_{2 \ell}} \cdots t_\ell^{\alpha_{\ell+1}}
 t_\ell^{\alpha_{\ell}}  (\frac qp)^{\alpha_{\ell}}  \al ^{\zeta_\ell
\alpha_{\ell}} \cdots t_1^{\alpha_{1}} (\frac q p)^{\alpha_{1} }   \al ^{{\zeta_1}
\alpha_{1}} =\\
&
 (\frac qp)^{\alpha_{1}+\dots+\alpha_{\ell} }
 t_1^{\alpha_1} \cdots t_\ell^{\alpha_\ell}
 t_\ell^{\alpha_{\ell+1}}
 \cdots t_1^{\alpha_{2\ell}}
 \al ^{\zeta_{1} \alpha_{1}+\dots+\zeta_{\ell }
 \alpha_{\ell} }
 \end{aligned}
$$
Since
$$
-(\zeta_{\ell+1} \alpha_{\ell+1}+\dots+\zeta_{2\ell} \alpha_{2 \ell}) =\zeta_{1} \alpha_{1}+\dots+\zeta_{\ell }\alpha_\ell,
$$
we obtain
$ \theta(\kappa [\alpha]- [\hat \alpha] )=0.
$
Thus, $\kappa [\alpha]- [\hat \alpha]\in \ker(\theta) $ yielding
 $\kappa [\alpha]- [\hat \alpha]\in \mathcal{I}$.

The proof of the inclusion $\mathcal{I}\subset I_{Sib} $ is similar
as the proof in Theorem 5.2.2 of \cite{RS}.
\end{proof}

As an immediate corollary of  Theorem  \ref{t.isym} and Proposition \ref{prTR}
we have
\begin{theorem} \label{thm1}
 {
The variety of the Sibirsky ideal $I_{Sib} $ is the Zariski closure of the set
$\mathcal{R}$ of all time-reversible systems in the family
(\ref{gs_uv}).}
\end{theorem}

The above studies shows that the theory regarding to the  computation and the structure
 of the saddle quantities for $1:-1$ resonant systems
has a counterpart also in the family of   $p:-q$ resonant systems,
however not in the whole family, but just in subfamilies of the form
\eqref{gs_uv_inf} and \eqref{gs_uv}. It is in agreement with the known
fact   that the study of local integrability of $p:-q$ resonant systems is
much more difficult that the studies in $1:-1$ case, which can be
observed already in the quadratic and the cubic case \cite{CGRS,FSZ}.

\section*{Acknowledgements}
 The first
author  is partially supported by a MINECO/ FEDER grant number PID2020-113758GB-I00 and an AGAUR (Generalitat
de Ca\-ta\-lu\-nya) grant number 2017SGR-1276. The second author acknowledges the support by  Slovenian Research Agency (core research program P1-0306).


\begin{thebibliography}{9}

\small

\bibitem{AGG} A. Algaba, C. Garc\'ia, J. Gin\'e, {\it Orbital Reversibility of Planar Vector Fields}, Mathematics {\bf 9} (2021), no. 14, https://doi.org/10.3390/math9010014

\bibitem{BBT} J.L.R. Bastos, C.A. Buzzi, J. Torregrosa, {\it Orbitally symmetric systems with application to planar centers}, Commun. Pure Appl. Anal. {\bf 20} (2021), no. 10, 3319--3346.

\bibitem{Bib} Y. N. Bibikov, Local Theory of Nonlinear Analytic Ordinary Differential Equations, Lecture Notes in Mathematics, Vol.~702, New York: Springer-Verlag, 1979.
    
\bibitem{CGRS} X. Chen, J. Gin\'e, V.G. Romanovski, D. Shafer, {\it The $1:-q$ resonant center problem for certain cubic Lotka-Volterra systems}, Appl. Math. Comput. {\bf 218} (2012), no. 23, 11620--11633.

\bibitem{CGMM} A. Cima, A. Gasull, F. Ma\~{n}osa, F. Manosas, {\it Algebraic properties of the Liapunov and period constants}, Rocky Mountain J. Math. {\bf 27} (1997), no. 2, 471--501.

\bibitem{CT} P. Conti, C. Traverso, Buchberger algorithm and integer programming, in Applied Algebra, Algebraic Algorithms, and Error-Correcting Codes AAECC'9 (H.F. Mattson, T. Mora, and T.R.N Rao eds.), Lecture Notes in Comput. Sci., vol. 539, Springer verlag, Berlin and New York, 1991, 130--139.

\bibitem{Cox} D. Cox, J. Little, D. O'Shea, {\it Ideals, Varieties, and Algorithms}, New York, Springer--Verlag, 1992.

\bibitem{FG} B. Fer\u{c}ec, J, Gin\'e, {\it Blow-up method to compute necessary conditions of integrability for planar differential systems}, Appl. Math. Comput. {\bf 358} (2019), 16--24. 

\bibitem{FSZ} A.~Fronville, A.~P.~Sadovski, H.~\.Zo\l\c{a}dek, {\it The solution of the $1 : -2$ resonant center problem in the quadratic case}, Fund.~Math. {\bf 157} (1998), 191--207.
    
\bibitem{GG} A. Gasull, J. Gin\'e, {\it Integrability of Li\'enard systems with a weak saddle}, Z. Angew. Math. Phys. {\bf 68} (2017), no. 1, paper no. 13, 7 pp.

\bibitem{GL} J. Gin\'e, J. Llibre, {\it On the integrability of Li\'enard systems with a strong saddle}, Appl. Math. Lett. {\bf 70} (2017), 39--45.

\bibitem{GM} J. Gin\'e, S. Maza, {\it The reversibility and the center problem}, Nonlinear Anal. {\bf 74} (2011),  no. 2, 695--704.

\bibitem{HPR}
M. Han, T. Petek, V. G. Romanovski,
{\it Reversibility in polynomial systems of ODE?s,
Appl. Math.  Comput. } {\bf 338} (2018) 55--71. 

\bibitem{JLR} A.~S.~Jarrah, R.~Laubenbacher, V.~Romanovski, {\it The Sibirsky component of the center variety of polynomial differential systems}, Computer algebra and computer analysis (Berlin, 2001), J. Symbolic Comput. {\bf 35} (2003), no. 5, 577--589.

\bibitem{Lamb} J.~S.~W.~Lamb and J.~A.~G.~Roberts, {\it Time-reversal symmetry in dynamical systems: a survey},
               Time-reversal symmetry in dynamical systems (Coventry, 1996),  Phys. D {\bf 112} (1998), no. 1-2, 1--39.

\bibitem{Liu2} Y.R. Li, J.B. Li, {\it Theory of values of singular point in complex autonomous differential systems}, Sci.~China Ser. A  {\bf 33} (1989), no. 1, 10--23.

\bibitem{LPW} J. Llibre, C. Pantazi, S. Walcher, First integrals of local analytic differential systems, Bull. Sci. Math. {\bf 136} (2012), 342--359.

\bibitem{MZ} D. Montgomery, L. Zippin, Topological Transformations Groups; Interscience: New York, NY, USA, 1955.

\bibitem{R} V.~G.~Romanovski, {\it Time-reversibility in 2-dimensional systems}, Open Syst. Inf. Dyn. {\bf 15} (2008), no. 4, 359--370.

\bibitem{RS1} V.~G.~Romanovski, D.~S.~Shafer, {\it Time-reversibility in two-dimensional polynomial systems}, In:
              \emph{Trends in Mathematics, Differential Equations with Symbolic Computations} (D.~Wang and Z.~Zheng, Eds.), 67--84. Basel, Birkhauser Verlag, 2005.

\bibitem{RS_BMS} V. G. Romanovski, D.S. Shafer. {\it   On the center problem for $p:-q$ resonant polynomial vector fields.} Bull. Belg. Math. Soc. Simon Stevin {\bf 15} (2008), no. 5,  871--887.

\bibitem{RS} V. G. Romanovski, D. S. Shafer, {\it The Center and Cyclicity Problems. A Computational Algebra Approach}, Birkh\"auser, Boston-Basel-Berlin, 2009.

\bibitem{RXZ} V.~G.~Romanovski, Y.~Xia, X.~Zhang, {\it Varieties of local integrability of analytic differential systems and their applications}, J. Differential Equations {\bf 257} (2014), 3079--3101.

\bibitem{Sib1} K.S. Sibirsky, {\it Algebraic Invariants of Differential Equations and Matrices},                 Kishinev, Shtiintsa (in Russian), 1976.
    
\bibitem{Sib2} K.S. Sibirsky, {\it Introduction to the Algebraic Theory of Invariants of  Differential Equations} Kishinev: Shtiintsa (in Russian), 1982. English transl.:  Manchester University Press, 1988.

\bibitem{St} B. Sturmfels, {\it Algorithms in Invariant Theory}, Springer-Verlag, New York, 1993.

\bibitem{TM11} M.A. Teixeira, R.M. Martins, {\it  Reversible-equivariant systems and matricial equations}, 
 An. Acad. Brasil. Cienc. {\bf 83} (2011), no. 2, 375--390. 

\bibitem{W} S. Walcher, {\it On transformations into normal form}, J. Math. Anal. Appl. {\bf 180} (1993), 617--632.

\bibitem{WRZ} L. Wei, V.G. Romanovski, X. Zhang, {\it Generalized involutive symmetry and its application in integrability of differential systems},  Z. Angew. Math. Phys. {\bf 68} (2017), no. 6, paper no. 132, 21 pp. 
    
\bibitem{Z} X. Zhang, {\it Analytic normalization of analytic integrable systems and the embedding flows}, J. Differential Equations {\bf 250} (2008), 1080--1092.

\end{thebibliography}
\end{document}